\documentclass[twoside,english]{amsart}
\usepackage[T1]{fontenc}
\usepackage{textcomp}
\usepackage[latin9]{inputenc}
\usepackage{mathtools}
\usepackage{amsbsy}
\usepackage{amstext}
\usepackage{amsthm}
\usepackage{amssymb}
\usepackage{geometry}
\usepackage[colorlinks,linkcolor=blue,citecolor=blue, pdfstartview=FitH, pagebackref]{hyperref}
\makeatletter

\usepackage{amscd}
\usepackage{easybmat}
\usepackage{mathrsfs}
\usepackage{amsfonts}
\usepackage{color}
\usepackage{pifont}
\usepackage{upgreek}
\usepackage{bm}
\usepackage{hyperref}
\usepackage{shorttoc}
\usepackage{amsmath,amstext,amsthm,amssymb,amscd}
\usepackage[mathscr]{eucal}
\usepackage{mathrsfs}
\usepackage{epsf}
\usepackage{tikz}
\usepackage{CJK}

\numberwithin{equation}{section}

\def\o{\overline}
\def\b{\bar}

\def\mr{\mathrm}

\def\n{\nabla}

\def\wt{\widetilde}

\def\End{\mathrm{End}}

\def\sgn{\mathrm{sgn}}

\def\tr{\mathrm{tr}}

\theoremstyle{plain}
\newtheorem{thm}{Theorem}[section]
\newtheorem{lemma}[thm]{Lemma}
\newtheorem{prop}[thm]{Proposition}

\theoremstyle{definition}
\newtheorem{rem}[thm]{Remark}

\newtheorem{que}[thm]{Question}
\newtheorem{pro}[thm]{Open problem}

\theoremstyle{definition}
\newtheorem{defn}[thm]{Definition}

\newcommand{\comment}[1]{}

\usepackage{amscd}

\usepackage{fancyhdr}
\makeatletter
\@namedef{subjclassname@2020}{2020 Mathematics Subject Classification}

\begin{document}
\begin{CJK}{UTF8}{gbsn}

\title{Positivity of the third Chern form for Griffiths positive vector bundles}

\author{Xueyuan Wan}

\address{Xueyuan Wan: Mathematical Science Research Center, Chongqing University of Technology, Chongqing 400054, China}
\email{xwan@cqut.edu.cn}

\begin{abstract}
In this paper, we prove the positivity of the double mixed discriminant associated with a positive linear map between spaces of third-order complex matrices, thereby settling the three-dimensional case of Finski's open problem. As an application, we obtain the weak positivity of the third Chern form for Griffiths positive vector bundles. Moreover, we show that all Schur forms are weakly positive for Griffiths positive vector bundles of rank three over complex threefolds. This yields a complete affirmative answer, in the case where both the rank and the dimension are three, to the question posed by Griffiths in 1969.
\end{abstract}

\thanks{Research of Xueyuan Wan is sponsored by the National Key R\&D Program of China (Grant No. 2024YFA1013200) and the Natural Science Foundation of Chongqing, China (Grant No. CSTB2024NSCQ-LZX0040, CSTB2023NSCQ-LZX0042).}

\maketitle
\section*{Introduction}

A central problem in complex differential geometry is to understand how \emph{curvature positivity} of a Hermitian holomorphic vector bundle constrains the \emph{positivity of its characteristic forms}.
Let \((E,h^E)\) be a Hermitian holomorphic vector bundle of rank \(r\) over a complex manifold \(X\) of complex dimension \(n\).
Its Chern forms \(c_i(E,h^E)\in A^{i,i}(X)\) (\(0\le i\le r\)) are defined by
\[
c(E,h^E):=\sum_{i=0}^r c_i(E,h^E)
:=\det\!\left(\mr{Id}_E+\frac{\sqrt{-1}}{2\pi}R^E\right),
\]
where \(R^E\in A^{1,1}(X,\mr{End}(E))\) denotes the Chern curvature of \((E,h^E)\), and \(c_i(E,h^E)\) is the \((i,i)\)-component of \(c(E,h^E)\).

Fix \(1\le k\le n\), and let \(\lambda=(\lambda_1,\dots,\lambda_r)\) be a partition of \(k\) (allowing trailing zeros) with
\[
r\ge \lambda_1\ge \cdots\ge \lambda_r\ge 0,
\qquad
|\lambda|:=\sum_{i=1}^r\lambda_i=k.
\]
To \(\lambda\) one associates the \emph{Schur form}
\[
P_\lambda\bigl(c(E,h^E)\bigr)
:=\det\bigl(c_{\lambda_i-i+j}(E,h^E)\bigr)_{1\le i,j\le r},
\]
a closed real \((k,k)\)-form (with the conventions \(c_0=1\) and \(c_m=0\) for \(m<0\) or \(m>r\)).
Schur forms simultaneously generalize Chern forms and (signed) Segre forms, and thus provide a unified family of characteristic forms naturally attached to \((E,h^E)\).
For instance, when \(r=3\) and \(k=3\), there are exactly three Schur forms:
\begin{equation}\label{eqn4}
   \begin{cases}
 P_{(1,1,1)}(c(E,h^E))	&=c_1(E,h^E)^{\wedge 3}-2c_1(E,h^E)\wedge c_2(E,h^E)+c_3(E,h^E), \\
  P_{(2,1,0)}(c(E,h^E))	&=c_1(E,h^E)\wedge c_2(E,h^E)-c_3(E,h^E),\\
   P_{(3,0,0)}(c(E,h^E))&=c_3(E,h^E).
 \end{cases}
\end{equation}

On the numerical side, Bloch--Gieseker~\cite{MR297773} proved numerical positivity of Chern classes for ample vector bundles, and Fulton--Lazarsfeld~\cite[Theorem~I]{FL} showed that \emph{all} Schur polynomials are numerically positive for ample vector bundles.
Griffiths proposed an analytic counterpart: for bundles with Griffiths positive curvature, the corresponding Schur forms should be (weakly) positive as differential forms.
A convenient formulation is the following; see \cite[p.~1541, Question of Griffiths]{Fin}.

\begin{que}[Griffiths~\cite{Griffiths}]\label{Question}
Let \(P\in \mathbb R[c_1,\dots,c_r]\) be a nonzero nonnegative linear combination of Schur polynomials of weighted degree \(k\).
Is the form \(P\bigl(c_1(E,h^E),\dots,c_r(E,h^E)\bigr)\) weakly positive for every Griffiths positive vector bundle \((E,h^E)\) over a complex manifold \(X\) of dimension \(n\ge k\)?
\end{que}

We recall that a real \((k,k)\)-form \(u\) is \emph{weakly positive} (resp.\ \emph{weakly nonnegative}) if
\[
u\wedge (\sqrt{-1})^{(n-k)^2}\beta\wedge \overline{\beta}>0
\quad (\text{resp. }\ge 0)
\]
for every nonzero decomposable \((n-k,0)\)-form \(\beta=\beta_1\wedge\cdots\wedge\beta_{n-k}\); see Definition~\ref{various positivity}.

Griffiths~\cite[p.~249]{Griffiths} proved positivity of the second Chern form for Griffiths positive bundles, and G\"uler~\cite[Theorem~1.1]{MR2932990} established Griffiths' question for all signed Segre forms.
Diverio--Fagioli~\cite{DF} obtained weak positivity for further polynomials in Chern forms via pushforward formulas on flag bundles; see also~\cite{Fag20,Fag22}.
For stronger curvature positivity notions, more is known:
Bott--Chern~\cite[Lemma~5.3,(5.5)]{MR185607} proved nonnegativity of all Chern forms under Bott--Chern nonnegativity, and Li~\cite[Proposition~3.1]{MR4263677} extended this to all Schur forms.
Finski showed that Bott--Chern nonnegativity is equivalent to dual Nakano nonnegativity~\cite[Theorem~2.15]{Fin}, proved nonnegativity of all Schur forms for Nakano nonnegative bundles by an algebraic argument~\cite[Section~3.4]{Fin}, and obtained positivity for (dual) Nakano positive bundles via a differential-form refinement of the Kempf--Laksov formula~\cite[Theorem~1.1]{Fin}.
More recently, the author~\cite{MR4686290} proved positivity of Schur forms under the (weaker) assumption of strong decomposable positivity.

A major insight of Finski~\cite[Theorem~1.3]{Fin} is that Griffiths' Question~\ref{Question} is equivalent to a structural positivity statement for \emph{positive linear maps} expressed in terms of mixed discriminants.
Let \(V\) be a complex vector space of dimension \(r\).
The mixed discriminant
\(
\mathrm{D}_V:\End(V)^{\otimes r}\to\mathbb C
\)
is the polarization of the determinant introduced by Alexandrov~\cite[Section~1]{MR1597}.
Using the natural identification \(\End(V)\simeq \End(V)^*\), let \(\mathrm{D}_V^*:\mathbb C\to \End(V)^{\otimes r}\) be the dual map.
Finski posed the following question and proved that it is equivalent to Griffiths' Question~\ref{Question}; see \cite[Theorem~1.3]{Fin}.
(For a fixed rank \(r=\mathrm{rank}(E)=\dim V\), however, Problem~\ref{pro1} and Question~\ref{Question} need not be equivalent.)

\begin{pro}[Finski~\cite{Fin}]\label{pro1}
Let \(H:\End(V)\to \End(W)\) be a positive linear map, and assume \(\dim V=\dim W\).
Is it true that
\[
\mathrm{D}_W\circ H^{\otimes \dim W}\circ \mathrm{D}_V^*
\in \mathbb R
\]
is positive?
\end{pro}

Problem~\ref{pro1} remains open in general, although it is known under additional assumptions; see, for instance, \cite[Proposition~4]{Fin}.
In particular, the case \(\dim V=\dim W=2\) admits an affirmative answer, and thus the first genuinely new case is \(\dim V=\dim W=3\).

In rank \(3\), Fagioli~\cite[Theorem~0.1]{Fag20} proved weak nonnegativity of
\[
P_{(2,1,0)}\bigl(c(E,h^E)\bigr)=c_1(E,h^E)\wedge c_2(E,h^E)-c_3(E,h^E)
\]
for Griffiths semipositive bundles, and established the inequalities
\[
c_1(E,h^E)^{\wedge 3}\ \ge\ c_1(E,h^E)\wedge c_2(E,h^E)\ \ge\ c_3(E,h^E).
\]
By twisting with a line bundle of sufficiently small negative curvature, this can be upgraded to weak positivity of \(P_{(2,1,0)}(c(E,h^E))\) for Griffiths positive bundles of rank three (see Proposition~\ref{prop:P210-positive}).
In addition, G\"uler~\cite[Theorem~1.1]{MR2932990} proved that all signed Segre forms are weakly positive; in rank three this yields weak positivity of the Schur forms corresponding to \((1,0,0)\), \((1,1,0)\), and \((1,1,1)\).
Finally, the second Chern form \(c_2(E,h^E)=P_{(2,0,0)}(c(E,h^E))\) is also weakly positive; see \cite{Griffiths,Fag20} and Proposition~\ref{prop:P210-positive}.
Consequently, when \(\mathrm{rank}(E)=\dim X=3\), Griffiths' Question~\ref{Question} reduces to proving weak positivity of the top Chern form \(c_3(E,h^E)=P_{(3,0,0)}(c(E,h^E))\).
Combining this reduction with \cite[Theorem~3.4 and Corollary~3.5]{Fin}, we conclude that, in the case \(r=\mathrm{rank}(E)=\dim V=3\) and \(\dim X=n=3\), Problem~\ref{pro1} and Question~\ref{Question} are equivalent.

\medskip

\noindent\textbf{Main results.}
In this paper we settle the first nontrivial case \(\dim V=\dim W=3\) of Finski's Problem~\ref{pro1} and, equivalently, give a complete affirmative answer to Griffiths' Question~\ref{Question} for Griffiths positive bundles of rank \(3\) on complex threefolds.
Write \(H(E_{i\bar j})=B_{i\bar j}\in \End(W)\) for the images of the matrix units.
Then the double mixed discriminant
\[
\Phi:=\mathrm{D}_W\circ H^{\otimes r}\circ \mathrm{D}_V^*(1)
\]
admits the concrete expression
\[
\Phi=\sum_{\sigma\in S_r}\sgn(\sigma)\,
D\bigl(B_{1\overline{\sigma(1)}},\dots,B_{r\overline{\sigma(r)}}\bigr),
\]
where \(D(\cdot)\) denotes the mixed discriminant on \(\End(W)\).
The positivity of \(H\) is equivalent to the Griffiths-type condition that
\begin{equation}\label{eqn1}
\sum_{i,j=1}^r B_{i\bar j}\xi^i\overline{\xi^j}\succ 0
\qquad\text{for every }0\ne \xi=(\xi^1,\dots,\xi^r)\in \mathbb C^r.
\end{equation}
Following Gurvits~\cite[Definition~4.2]{MR2087945}, we consider the operator scaling
\[
S_{C_1,C_2}(H)(X):=C_1\,H(C_2^*XC_2)\,C_1^*,
\]
where \(C_1\in \End(W)\) and \(C_2\in \End(V)\) are invertible.
This scaling preserves positivity and does not change the sign of \(\Phi\).
Moreover, by \cite[Theorem~4.7]{MR2087945}, one can choose \(C_1,C_2\) so that \(S_{C_1,C_2}(\frac{1}{r}H)\) becomes \emph{doubly stochastic}.
Consequently, without loss of generality we work under the normalization
\begin{equation}\label{normalization}
\sum_{i=1}^r B_{i\bar i}=rI,
\qquad
\tr(B_{i\bar j})=r\,\delta_{ij}.
\end{equation}
Under \eqref{eqn1} and \eqref{normalization}, Problem~\ref{pro1} reduces to proving \(\Phi>0\).

\begin{thm}\label{thm:main}
Let \(B_{i\bar j}\in M_3(\mathbb C)\), \(1\le i,j\le 3\), satisfy \eqref{eqn1} and \eqref{normalization}.
Then
\[
\Phi=\sum_{\sigma\in S_3}\sgn(\sigma)\,
D\bigl(B_{1\overline{\sigma(1)}},B_{2\overline{\sigma(2)}},B_{3\overline{\sigma(3)}}\bigr)>0.
\]
In particular, \(\mathrm{D}_W\circ H^{\otimes \dim W}\circ \mathrm{D}_V^*\in \mathbb R\) is positive if \(H:\End(V)\to \End(W)\) is a positive linear map and \(\dim V=\dim W=3\).
Moreover, all Schur forms are weakly positive for Griffiths positive vector bundles of rank three over complex threefolds.
\end{thm}

\begin{rem}\label{rem01}
By \cite[Proposition~2.21]{Fin}, the case \(\dim V=\dim W=3\) is the first dimension in which a positive map is not necessarily decomposable in the sense of \cite[(2.40)]{Fin} (as shown by Choi's example~\cite{MR379365}).
Thus Theorem~\ref{thm:main} verifies Finski's open problem in the first genuinely new dimension and provides a strong indication that the general statement may also hold.
\end{rem}

In particular, Theorem~\ref{thm:main} implies that the third Chern form \(c_3(E,h^E)\) is weakly positive when \((E,h^E)\to X\) is Griffiths positive and \(\mathrm{rank}(E)=\dim X=3\).
We next extend this conclusion to arbitrary rank \(r\ge 3\) on any complex manifold \(X\) with \(\dim X\ge 3\).
Indeed, \(c_3(E,h^E)\) is the sum of all \(3\times 3\) principal minors of the Chern curvature matrix \((R^E)^i_{\,j}\):
\begin{equation}\label{eqn6}
c_3(E,h^E)
=\left(\tfrac{\sqrt{-1}}{2\pi}\right)^3
\sum_{1\le i_1<i_2<i_3\le r}
\det\bigl((R^E)^i_{\,j}\bigr)_{i,j\in\{i_1,i_2,i_3\}}.
\end{equation}
Moreover, weak positivity of a real \((p,p)\)-form can be tested by restricting to complex \(p\)-dimensional subspaces; see, for example, \cite[Proposition~1.9]{Fag20}.
Hence we may assume \(\dim X=3\).
Restricting \((E,h^E)\) to the rank-three subbundle spanned (locally) by \(\{e_{i_1},e_{i_2},e_{i_3}\}\) preserves Griffiths positivity, and applying Theorem~\ref{thm:main} shows that each summand in \eqref{eqn6} is weakly positive.
Therefore their sum \(c_3(E,h^E)\) is weakly positive.
We thus obtain:

\begin{thm}\label{thm11}
Let \((E,h^E)\) be a Griffiths positive vector bundle of rank \(r\ge 3\) over a complex manifold \(X\) with \(\dim X\ge 3\).
Then the third Chern form \(c_3(E,h^E)\) is weakly positive.
\end{thm}

\begin{rem}\label{rem1}
Our result yields an elementary proof of the following fact: if \(X\) is a compact complex threefold admitting a Hermitian metric with negative holomorphic bisectional curvature, then its Euler characteristic is negative.
Indeed,
\[
\chi(X)=\int_X c_3(T_X)=-\int_X c_3(T_X^*),
\]
and negative holomorphic bisectional curvature implies that the cotangent bundle \(T_X^*\) is Griffiths positive.
Hence \(c_3(T_X^*)\) is weakly positive, so \(\int_X c_3(T_X^*)>0\), and therefore \(\chi(X)<0\).
\end{rem}

\noindent\textbf{Proof idea and organization of the paper}
\paragraph{{\it Proof idea.}}
The proof of Theorem~\ref{thm:main} combines a normalization-by-scaling step with a geometric positivity argument for a mixed-discriminant expression.
We first use operator scaling to reduce to the doubly stochastic normalization~\eqref{normalization}, which preserves positivity and the sign of \(\Phi\).
Under this normalization, \(\Phi\) can be rewritten in terms of mixed discriminants of explicit third-order Hermitian matrices associated to \((B_{i\bar j})\).
Introducing \(C(\xi)=\sum_{i,j}B_{i\bar j}\,\xi^i\overline{\xi^j}\), a spherical moment identity expresses the relevant mixed-discriminant combinations as an average over \(\xi\in S^{5}\) of a symmetric polynomial in \(C(\xi)\).
Diagonalizing \(C(\xi)\) reduces the integrand to an inequality for symmetric polynomials of the eigenvalues of \(C(\xi)\), which follows from a sharp Schur inequality.
This yields \(\Phi>0\).
Finally, the equivalence of Finski's problem with Griffiths' question in the \((3,3)\)-case translates \(\Phi>0\) into weak positivity of all degree-\(3\) Schur forms, and Theorem~\ref{thm11} follows by restriction to rank-\(3\) subbundles and the principal-minor formula~\eqref{eqn6}.

\paragraph{{\it Organization of the paper.}}
In Section~\ref{sec1} we recall weak positivity of differential forms and Griffiths positivity for Hermitian holomorphic vector bundles.
In Section~\ref{sec2} we reformulate Finski's open problem as Open Problem~\ref{problem2} and solve it for \(r=3\); in particular, we prove Theorem~\ref{thm:main}.

\textbf{Acknowledgments.}
The author is grateful to Siarhei Finski for helpful discussions concerning Remark~\ref{rem01}, and to Vamsi Pritham Pingali for helpful discussions related to Theorem~\ref{thm11} and Remark~\ref{rem1}.

\section{Positivity of differential forms and vector bundles}\label{sec1}

In this section we recall the notions of positivity for differential forms and the definition of Griffiths positivity for Hermitian holomorphic vector bundles.

\subsection{Positivity of differential forms}
We first review several positivity notions for differential forms. For further background, see \cite[Section~1.1]{Fag20} and \cite{RA,Fin}.

Let \(V\) be a complex vector space of dimension \(n\), and let \(\{e_1,\dots,e_n\}\) be a basis of \(V\). Denote by \(\{e^1,\dots,e^n\}\) the dual basis of \(V^*\). Let \(\Lambda^{p,q}V^*\) be the space of \((p,q)\)-forms, and let \(\Lambda^{p,p}_{\mathbb R}V^*\subset \Lambda^{p,p}V^*\) be the subspace of real \((p,p)\)-forms.

\begin{defn}\label{various positivity}
A form \(\nu\in \Lambda^{n,n}V^*\) is called a non-negative (resp.\ positive) volume form if
\[
\nu=\tau\,(\sqrt{-1})^n\,e^1\wedge \overline{e^1}\wedge \cdots \wedge e^n\wedge \overline{e^n}
\]
for some \(\tau\in\mathbb R\) with \(\tau\ge 0\) (resp.\ \(\tau>0\)).
\end{defn}

Set \(q=n-p\). A \((q,0)\)-form \(\beta\) is called \emph{decomposable} if \(\beta=\beta_1\wedge\cdots\wedge \beta_q\) for some \(\beta_1,\dots,\beta_q\in V^*\).

\begin{defn}\label{defn1}
A real \((p,p)\)-form \(u\in \Lambda^{p,p}_{\mathbb R}V^*\) is called
\begin{itemize}
  \item \emph{weakly non-negative} (resp.\ \emph{weakly positive}) if for every nonzero decomposable \(\beta\in \Lambda^{q,0}V^*\), the form
  \[
  u\wedge (\sqrt{-1})^{q^2}\beta\wedge \overline{\beta}
  \]
  is a non-negative (resp.\ positive) volume form;
  \item \emph{non-negative} (resp.\ \emph{positive}) if for every nonzero \(\beta\in \Lambda^{q,0}V^*\), the form
  \[
  u\wedge (\sqrt{-1})^{q^2}\beta\wedge \overline{\beta}
  \]
  is a non-negative (resp.\ positive) volume form;
  \item \emph{strongly non-negative} (resp.\ \emph{strongly positive}) if there exist decomposable forms \(\alpha_1,\dots,\alpha_N\in \Lambda^{p,0}V^*\) such that
  \[
  u=\sum_{s=1}^N (\sqrt{-1})^{p^2}\alpha_s\wedge \overline{\alpha_s}.
  \]
\end{itemize}
\end{defn}

\begin{rem}
Let \(\mathrm{WP}^pV^*\), \(\mathrm{P}^pV^*\), and \(\mathrm{SP}^pV^*\) denote the closed convex cones in \(\Lambda^{p,p}_{\mathbb R}V^*\) generated by weakly non-negative, non-negative, and strongly non-negative forms, respectively. Then
\[
\mathrm{SP}^pV^*\subseteq \mathrm{P}^pV^*\subseteq \mathrm{WP}^pV^*.
\]
These inclusions are equalities for \(p=0,1,n-1,n\), and they are strict for \(2\le p\le n-2\); see, for instance, \cite[Remarks~1.7 and~1.8]{Fag20} and \cite{RA}.
\end{rem}

Let \(X\) be a complex manifold of dimension \(n\), and let \(A^{p,q}(X)\) denote the space of smooth \((p,q)\)-forms on \(X\).

\begin{defn}
A real \((p,p)\)-form \(\alpha\in A^{p,p}(X)\) is called weakly non-negative (resp.\ weakly positive), non-negative (resp.\ positive), or strongly non-negative (resp.\ strongly positive) if for every \(x\in X\), the form
\[
\alpha_x\in \Lambda^{p,p}_{\mathbb R}(T_x^{1,0}X)^*
\]
is weakly non-negative (resp.\ weakly positive), non-negative (resp.\ positive), or strongly non-negative (resp.\ strongly positive), respectively.
\end{defn}

\subsection{Griffiths positive vector bundles}
We next recall the Chern connection and its curvature for a Hermitian holomorphic vector bundle; see \cite[Chapter~1]{Kobayashi+1987} for details. Throughout we use the Einstein summation convention.

Let \(\pi:(E,h^E)\to X\) be a Hermitian holomorphic vector bundle over a complex manifold \(X\), where \(\mr{rank}\,E=r\) and \(\dim X=n\). The Chern connection \(\n^E\) of \((E,h^E)\) preserves \(h^E\) and is of type \((1,0)\). Given a local holomorphic frame \(\{e_i\}_{1\le i\le r}\) of \(E\), the connection satisfies
\[
\n^E e_i=\theta_i^{\,j}\,e_j,
\]
where \(\theta=(\theta_i^{\,j})\) is the connection \(1\)-form. More explicitly,
\[
\theta_i^{\,j}=\partial h_{i\bar k}\,h^{\bar k j},
\qquad
h_{i\bar k}:=h^E(e_i,e_k).
\]

The Chern curvature \(R^E=(\n^E)^2\in A^{1,1}(X,\mr{End}(E))\) can be written as
\[
R^E=R_i^{\,j}\,e_j\otimes e^i,
\]
where \(R=(R_i^{\,j})\) is the curvature matrix with entries in \(A^{1,1}(X)\), and \(\{e^i\}_{1\le i\le r}\) is the dual frame. One has
\[
R_i^{\,j}=d\theta_i^{\,j}+\theta_i^{\,k}\wedge \theta_k^{\,j}
=\bar\partial\,\theta_i^{\,j}.
\]

Let \(\{z^\alpha\}_{1\le \alpha\le n}\) be local holomorphic coordinates on \(X\). Then
\[
R_i^{\,j}=R_{i\alpha\bar\beta}^{\,j}\,dz^\alpha\wedge d\bar z^\beta,
\]
and, lowering an index using the metric,
\[
R_{i\bar j}:=R_i^{\,k}\,h_{k\bar j}
=R_{i\bar j\alpha\bar\beta}\,dz^\alpha\wedge d\bar z^\beta.
\]
In terms of the local metric coefficients, we have
\[
R_{i\bar j\alpha\bar\beta}
=R_{i\alpha\bar\beta}^{\,k}\,h_{k\bar j}
=-\partial_\alpha\partial_{\bar\beta}h_{i\bar j}
+h^{\bar\ell k}\,(\partial_\alpha h_{i\bar\ell})(\partial_{\bar\beta}h_{k\bar j}),
\]
where \(\partial_\alpha:=\partial/\partial z^\alpha\) and \(\partial_{\bar\beta}:=\partial/\partial \bar z^\beta\).

\begin{defn}\label{G-N}
A Hermitian holomorphic vector bundle \(\pi:(E,h^E)\to X\) is called \emph{Griffiths positive} (resp.\ \emph{Griffiths negative}) if
\[
R_{i\bar j\alpha\bar\beta}\,v^i\overline{v^j}\,\xi^\alpha\overline{\xi^\beta}>0
\quad(\text{resp. }<0)
\]
for all nonzero \(v=v^ie_i\in E|_z\) and \(\xi=\xi^\alpha \frac{\partial}{\partial z^\alpha}\in T^{1,0}_zX\), at every point \(z\in X\).
\end{defn}

\section{Positivity of Schur forms}\label{sec2}

In this section, we reformulate Finski's open problem in a more concrete form. We then prove that for Griffiths positive vector bundles of ranks two and three, all Schur forms are weakly positive.

\subsection{Finski's open problem}\label{Sec: Fin-open}

In order to study Griffiths' Question~\ref{Question}, Finski proposed a new open problem and proved that it is equivalent to Griffiths' question.

We now state Finski's open problem. Let \(V\) be a complex vector space with \(\dim V=r\). The mixed discriminant
\[
\mathrm{D}_V:\End(V)^{\otimes r}\longrightarrow \mathbb C
\]
was introduced by Alexandrov \cite[Section~1]{MR1597} as the polarization of the determinant. In coordinates, for matrices \(A^i=(a^i_{k\b{l}})_{k,l=1}^r\) (\(i=1,\dots,r\)), one has
\begin{equation}\label{mixed}
\mathrm{D}_V(A^1,\dots,A^r)
=\frac{1}{r!}\frac{\partial^r}{\partial t^1\cdots \partial t^r}\det\bigl(t^1A^1+\cdots+t^rA^r\bigr)
=\frac{1}{r!}\sum_{\sigma\in S_r}\det\bigl(a^{\sigma(i)}_{i\b{k}}\bigr)_{i,k=1}^r,
\end{equation}
where \(S_r\) is the symmetric group on \(\{1,\dots,r\}\).

Using the natural duality \(\mr{End}(V)\simeq \mr{End}(V)^*\), we denote by
\[
\mathrm{D}_V^*:\mathbb C\longrightarrow \mr{End}(V)^{\otimes r}
\]
the dual map of \(\mathrm{D}_V\).
\begin{defn}\label{positive map}
Given complex vector spaces \(V\) and \(W\) with \(\dim V=\dim W\), we say that a linear map
\(H:\End(V)\to \End(W)\) is \emph{non-negative} (resp.\ \emph{positive}) if it sends every nonzero positive semidefinite endomorphism of \(V\) to a positive semidefinite (resp.\ positive definite) endomorphism of \(W\).
\end{defn}

\begin{pro}[Finski \cite{Fin}]\label{problem1}
Let \(H:\End(V)\to \End(W)\) be positive, and assume \(\dim V=\dim W\).
Is it true that 
\[
\mathrm{D}_W\circ H^{\otimes \dim W}\circ \mathrm{D}_V^*\in \mathbb R
\]
 positive?
\end{pro}

Moreover, Finski proved that Open Problem~\ref{problem1} is equivalent to Griffiths' Question~\ref{Question}.

\begin{thm}[{\cite[Theorem~1.3]{Fin}}]\label{thm1}
The answers to Open Problem~\ref{problem1} and Griffiths' Question~\ref{Question} coincide.
\end{thm}

A key step in Finski's argument is a pushforward identity for characteristic forms, which refines the Kempf--Laksov determinantal formula for holomorphic vector bundles at the level of differential forms.
More precisely, he established the pushforward formula
\[
P_\lambda\bigl(c(E,h^E)\bigr)=\pi_*^\lambda\!\left[c_{\mathrm{rk}(Q)}(Q,h^Q)\right],
\]
where
\[
Q:=\bigl(\pi^*\mathrm{Hom}(V_X,E)\oplus \mathcal O\bigr)\big/\mathcal O_{\mathbb P_{\mathrm{Hom}}}(-1)
\quad\text{on}\quad
\mathbb P_{\mathrm{Hom}}:=\mathbb P\bigl(\mathrm{Hom}(V_X,E)\oplus \mathcal O\bigr).
\]
Consequently, the weak positivity of Schur forms is reduced to the positivity of the top Chern form of the quotient bundle \(Q\).
Expanding this top Chern form, one finds that its coefficients are precisely governed by the operator
\(\mathrm{D}_W\circ H^{\otimes \dim W}\circ \mathrm{D}_V^*\).

Now we study Finski's open problem, Problem~\ref{problem1}. By the definition of the mixed discriminant \eqref{mixed}, we have
\begin{align*}
\mathrm{D}_V(A^1,\dots,A^r)
&=\frac{1}{r!}\sum_{\sigma,\tau\in S_r}\sgn(\sigma)\sgn(\tau)\prod_{i=1}^r a^i_{\sigma(i)\overline{\tau(i)}} \\
&=\Bigl\langle A^1\otimes \cdots \otimes A^r,\;
\frac{1}{r!}\sum_{\sigma,\tau\in S_r}\sgn(\sigma)\sgn(\tau)\bigotimes_{i=1}^r E_{\sigma(i)\overline{\tau(i)}} \Bigr\rangle,
\end{align*}
where \(E_{i\bar j}\) denotes the \(r\times r\) matrix whose \((i,j)\)-entry is \(1\) and whose other entries are \(0\), and \(\langle\cdot,\cdot\rangle\) is the natural pairing on \(\End(V)^{\otimes r}\). It follows that
\[
\mathrm{D}_V^*(1)=\frac{1}{r!}\sum_{\sigma,\tau\in S_r}\sgn(\sigma)\sgn(\tau)\bigotimes_{i=1}^r E_{\sigma(i)\overline{\tau(i)}}.
\]

Let \(H:\End(V)\to \End(W)\) be a positive linear map, and write
\[
H(E_{i\bar j})=B_{i\bar j}\in \End(W).
\]
For a nonzero vector \(\xi=(\xi^1,\cdots,\xi^r)^\top\in \mathbb C^r\), consider the rank-one positive semidefinite matrix
\[
A=\xi\xi^*=\sum_{i,j=1}^r \xi^i\overline{\xi^j}\,E_{i\bar j}.
\]
Since \(H\) is positive, $H(A)$ is positive definite, that is,  
\[
H(A)=\sum_{i,j=1}^r \xi^i\overline{\xi^j}\,H(E_{i\bar j})
=\sum_{i,j=1}^r \xi^i\overline{\xi^j}\,B_{i\bar j}\succ 0.
\]
In particular, each \(B_{i\bar j}\) satisfies the Hermitian symmetry \(B_{i\bar j}^*=B_{j\bar i}\).

Conversely, assume that \(H(\xi\xi^*)\) is positive definite for every nonzero vector \(\xi\in\mathbb C^r\).
Let \(X\in M_r(\mathbb C)\) be a nonzero positive semidefinite matrix. Choose vectors \(\xi_1,\dots,\xi_r\in\mathbb C^r\), not all zero, such that
\[
X=(\xi_1,\dots,\xi_r)^{\top}\,(\overline{\xi}_1,\dots,\overline{\xi}_r)
=\sum_{i,j=1}^r (\xi_i^{\top}\overline{\xi}_j)\,E_{i\bar j}
=\sum_{i,j,k=1}^r \xi_i^{k}\,\overline{\xi_j^{k}}\,E_{i\bar j}.
\]
Since \(X\neq 0\), there exists an index \(k\) such that the \(k\)-th row \((\xi_1^{k},\dots,\xi_r^{k})\) is not the zero vector. Therefore,
\[
H(X)
=\sum_{i,j,k=1}^r \xi_i^{k}\,\overline{\xi_j^{k}}\,H(E_{i\bar j})
=\sum_{k=1}^r \left(\sum_{i,j=1}^r \xi_i^{k}\,\overline{\xi_j^{k}}\,B_{i\bar j}\right)\succ 0,
\]
where we write \(B_{i\bar j}:=H(E_{i\bar j})\). In particular, \(H(X)\) is positive definite for every nonzero positive semidefinite \(X\), i.e.\ \(H\) is positive.

We now set
\begin{align}\label{eq:defPhi}
\begin{split}
\Phi
&:=\mathrm{D}_W\circ H^{\otimes r}\circ \mathrm{D}_V^*(1) \nonumber\\
&=\frac{1}{r!}\mathrm{D}_W\!\left(\sum_{\sigma,\tau\in S_r}\sgn(\sigma)\sgn(\tau)\bigotimes_{i=1}^r B_{\sigma(i)\overline{\tau(i)}}\right) \nonumber\\
&=\frac{1}{r!}\sum_{\sigma,\tau\in S_r}\sgn(\sigma)\sgn(\tau)\,
D\bigl(B_{\sigma(1)\overline{\tau(1)}},\dots,B_{\sigma(r)\overline{\tau(r)}}\bigr)\\
&=\sum_{\sigma\in S_r}\sgn(\sigma)\,
D\bigl(B_{1\overline{\sigma(1)}},\dots,B_{r\overline{\sigma(r)}}\bigr).
\end{split}
\end{align}

Next, denote by
\[
\theta_{i\bar j,p\bar q}:=(B_{i\bar j})_{p\bar q}
\]
the \((p,q)\)-entry of the matrix \(B_{i\bar j}\). We may form the two families of matrices
\[
A_{p\bar q}:=(\theta_{i\bar j,p\bar q})_{i\bar j}\in \End(V),
\qquad
B_{i\bar j}:=(\theta_{i\bar j,p\bar q})_{p\bar q}\in \End(W).
\]
In terms of \(\theta\), the quantity \(\Phi\) can be rewritten as
\begin{align*}
\Phi
&=\frac{1}{r!}\sum_{\sigma,\tau\in S_r}\sgn(\sigma)\sgn(\tau)\,
D\bigl(B_{\sigma(1)\overline{\tau(1)}},\dots,B_{\sigma(r)\overline{\tau(r)}}\bigr)\\
&=\frac{1}{(r!)^2}\sum_{\sigma,\tau,\gamma,\pi\in S_r}\sgn(\sigma)\sgn(\tau)\sgn(\gamma)\sgn(\pi)
\prod_{i=1}^r \theta_{\sigma(i)\overline{\tau(i)},\,\gamma(i)\overline{\pi(i)}}\\
&=\frac{1}{r!}\sum_{\gamma,\pi\in S_r}\sgn(\gamma)\sgn(\pi)\,
D\bigl(A_{\gamma(1)\overline{\pi(1)}},\dots,A_{\gamma(r)\overline{\pi(r)}}\bigr).
\end{align*}

Following \cite[Definition~4.2]{MR2087945}, Gurvits introduced the operator scaling of \(H\) by
\begin{equation}\label{scaling}
S_{C_1,C_2}(H)(X):=C_1\,H(C_2^*XC_2)\,C_1^*,
\end{equation}
where \(C_1\in \End(W)\) and \(C_2\in \End(V)\) are invertible matrices. Define
\[
\widetilde{\Phi}
:=\mathrm{D}_W\circ \bigl(S_{C_1,C_2}(H)\bigr)^{\otimes r}\circ \mathrm{D}_V^*(1).
\]
By multilinearity of the mixed discriminant and its covariance under conjugation (see, e.g., \cite[Fact~1]{MR2200852}), we obtain
\begin{align*} \begin{split}
 \wt{\Phi}: &=\frac{1}{r!}\sum_{\sigma,\tau\in S_r}\mathrm{sgn}(\sigma)\mathrm{sgn}(\tau)D(C_1 H(C_2^*E_{\sigma(1)\o{\tau(1)}}C_2)C_1^*,\cdots, C_1 H(C_2^*E_{\sigma(r)\o{\tau(r)}}C_2)C_1^*)\\ &=\frac{1}{r!}|\det(C_1)|^2\sum_{\sigma,\tau\in S_r}\mathrm{sgn}(\sigma)\mathrm{sgn}(\tau)D(H(C_2^*E_{\sigma(1)\o{\tau(1)}}C_2),\cdots, H(C_2^*E_{\sigma(r)\o{\tau(r)}}C_2))\\ &=\frac{1}{(r!)^2}|\det(C_1)|^2\sum_{\sigma,\tau,\gamma,\pi\in S_r}\mathrm{sgn}(\sigma)\mathrm{sgn}(\tau)\mathrm{sgn}(\gamma)\mathrm{sgn}(\pi)\sum_{k_1,\cdots,k_r,\atop l_1,\cdots,l_r=1}^r\prod_{i=1}^r(C_2^*)_{k_i\o{\sigma(i)}}(C_2)_{\tau(i)\o{l_i}}\theta_{k_i\o{l_i},\gamma(i)\o{\pi(i)}}\\ &=\frac{1}{r!}|\det(C_1)|^2\sum_{\gamma,\pi\in S_r}\mathrm{sgn}(\gamma)\mathrm{sgn}(\pi)D(C_2A_{\gamma(1)\o{\pi(1)}}C_2^*,\cdots, C_2A_{\gamma(r)\o{\pi(r)}}C_2^*)\\ 
&=|\det(C_1)|^2|\det (C_2)|^2\Phi, 
\end{split} \end{align*}
In particular, if both \(C_1\) and \(C_2\) are inverse matrices, then the scaling \eqref{scaling} does not change the sign of \(\Phi\). Moreover, \(S_{C_1,C_2}(H)\) remains a positive map.

This suggests working with a suitable scaling of \(H\). If \(H:\End(V)\to \End(W)\) is a positive linear map, then \(H(X)\) is positive definite (and hence \(\mathrm{rank}(H(X))=r\)) for every nonzero positive semidefinite matrix \(X\in \End(V)\) with \(1\le \mathrm{rank}(X)<r\). Consequently, \(H\) is indecomposable in the sense of \cite[Definition~2.3]{MR2087945}.
 By \cite[Theorem~4.7]{MR2087945} (see also \cite[Theorem 1.2]{Idel2016MatrixScalingReview}), there exist inverse matrices \(C_1\) and \(C_2\) such that \(S_{C_1,C_2}(H)\) is \emph{doubly stochastic}, i.e.
\[
S_{C_1,C_2}(H)(I)=I,
\qquad
S_{C_1,C_2}(H)^*(I)=I
\]
(see \cite[Definition~2.3]{MR2087945}), where \(I\) denotes the identity matrix. Since this scaling does not change the sign of \(\Phi\), we may assume without loss of generality that \(\frac{1}{r}H\) is doubly stochastic. Equivalently,
\[
H(I)=rI,
\qquad
H^*(I)=rI.
\]
The condition \(H(I)=rI\) is the same as
\(
\sum_{i=1}^r B_{i\bar i}=rI,
\)
and \(H^*(I)=rI\) is equivalent to \(\mathrm{tr}(H(X))=r\mathrm{tr}(X)\), hence to
\(
\mathrm{tr}(B_{i\bar j})=r\,\delta_{ij}.
\)

Therefore, Finski's open problem, Problem~\ref{problem1}, reduces to the following matrix formulation.

\begin{pro}\label{problem2}
Let \(B_{i\bar j}\in M_r(\mathbb C)\) (\(1\le i,j\le r\)) satisfy:
\begin{itemize}
  \item[(i)] \(\sum_{i,j=1}^r B_{i\bar j}\xi^i\overline{\xi^j}\succ 0\) for every \(0\ne \xi=(\xi^1,\dots,\xi^r)\in\mathbb C^r\);
    \item[(ii)] \(\sum_{i=1}^r B_{i\bar i}=rI\);
  \item[(iii)] \(\mathrm{tr}(B_{i\bar j})=r\,\delta_{ij}\).
\end{itemize}
Is the quantity
\[
\Phi=\sum_{\sigma\in S_r}\sgn(\sigma)\,
D\bigl(B_{1\overline{\sigma(1)}},\dots,B_{r\overline{\sigma(r)}}\bigr)
\]
positive?
\end{pro}

In the next subsections, we study Problem~\ref{problem2} for \(r=2,3\).

\subsection{The case of rank two}

For \(r=2\), the weak positivity of Schur forms was proved by Griffiths \cite{Griffiths}. Here we present an alternative proof. When \(r=2\), we have
\[
\Phi=D(B_{1\bar1},B_{2\bar2})-D(B_{1\bar2},B_{2\bar1}).
\]
By the definition of the mixed discriminant,
\[
2D(X,Y)=\tr(X)\tr(Y)-\tr(XY).
\]
Therefore, using the normalization \(\tr(B_{1\bar2})=0\), we obtain
\[
2D(B_{1\bar2},B_{2\bar1})
=\tr(B_{1\bar2})\tr(B_{2\bar1})-\tr(B_{1\bar2}B_{2\bar1})
=-\tr(B_{1\bar2}B_{2\bar1})
=-\|B_{1\bar2}\|^2,
\]
where \(\|B_{1\bar2}\|^2:=\tr(B_{1\bar2}B_{1\bar2}^*)=\tr(B_{1\bar2}B_{2\bar1})\). Hence
\[
\Phi=D(B_{1\bar1},B_{2\bar2})+\tfrac{1}{2}\|B_{1\bar2}\|^2\ge D(B_{1\bar1},B_{2\bar2})\ge 1,
\]
where the last inequality follows from \cite{MR2200852} on a proof of Bapat's conjecture for the complex case.

\medskip

We now give a second proof of \(\Phi\ge 1\) based on an integral representation, which does not rely on Gurvits' inequality. We will use the following lemma (see \cite[Lemma~7.24]{WatrousTQI_Draft2018}).

\begin{lemma}\label{lemma1}
Let \(V=\mathbb C^r\) be a complex Euclidean space of dimension \(r\), let \(n\) be a positive integer, and let \(\mu\) be the normalized rotation-invariant measure on the unit sphere \(S^{2r-1}\). Then
\[
\Pi_{V^{\odot n}}
=\dim(V^{\odot n})\int_{S^{2r-1}}(\xi\xi^*)^{\otimes n}\,d\mu(\xi),
\]
where \(V^{\odot n}\) denotes the \(n\)-th symmetric tensor power of \(V\), and \(\Pi_{V^{\odot n}}\) is the orthogonal projector onto \(V^{\odot n}\subset V^{\otimes n}\).
\end{lemma}

As a consequence, we obtain the following trace-moment identity.

\begin{lemma}\label{lemma:trace-moment}
For any matrices \(U_1,\dots,U_n\in M_r(\mathbb C)\), one has
\[
\int_{S^{2r-1}}\prod_{i=1}^n(\xi^*U_i\xi)\,d\mu(\xi)
=\frac{1}{(r)_n}\sum_{\pi\in S_n}\tr_\pi(U_1,\dots,U_n),
\]
where \((r)_n=r(r+1)\cdots(r+n-1)\), and
\[
\tr_\pi(U_1,\dots,U_n)=\prod_{\ell=1}^s \tr\!\left(\prod_{i\in c_\ell} U_i\right)
\]
if \(\pi=(c_1)\cdots(c_s)\) is the cycle decomposition of \(\pi\in S_n\).
\end{lemma}

\begin{proof}
By Lemma~\ref{lemma1},
\begin{align*}
\int_{S^{2r-1}}\prod_{i=1}^n(\xi^*U_i\xi)\,d\mu(\xi)
&=\frac{1}{\dim(V^{\odot n})}\,
\tr\!\left[\int_{S^{2r-1}}(\xi\xi^*)^{\otimes n}\,d\mu(\xi)\cdot (U_1\otimes\cdots\otimes U_n)\right]\\
&=\frac{1}{\dim(V^{\odot n})}\,\tr\!\bigl(\Pi_{V^{\odot n}}(U_1\otimes\cdots\otimes U_n)\bigr).
\end{align*}
Let $W_\sigma$ be the permutation operator induced by $\sigma\in S_n$. Since \(\Pi_{V^{\odot n}}=\frac{1}{n!}\sum_{\sigma\in S_n}W_\sigma\), and \(\dim(V^{\odot n})=\binom{r+n-1}{n}\), the claimed formula follows, with \((r)_n=n!\binom{r+n-1}{n}\).
\end{proof}

For \(r=2\), Lemma~\ref{lemma:trace-moment} gives
\[
\int_{S^{3}}(\xi^*U\xi)(\xi^*V\xi)\,d\mu(\xi)
=\frac{1}{6}\bigl(\tr(U)\tr(V)+\tr(UV)\bigr).
\]
Consequently,
\begin{align*}
2D(B_{1\bar1},B_{2\bar2})
&=\tr(B_{1\bar1})\tr(B_{2\bar2})-\tr(B_{1\bar1}B_{2\bar2})\\
&=2\tr(B_{1\bar1})\tr(B_{2\bar2})
-6\int_{S^3}(\xi^*B_{1\bar1}\xi)(\xi^*B_{2\bar2}\xi)\,d\mu(\xi),
\end{align*}
and similarly
\[
2D(B_{1\bar2},B_{2\bar1})
=-6\int_{S^3}(\xi^*B_{1\bar2}\xi)(\xi^*B_{2\bar1}\xi)\,d\mu(\xi),
\]
because \(\tr(B_{1\bar2})=\tr(B_{2\bar1})=0\).
Using \(\tr(B_{1\bar1})=\tr(B_{2\bar2})=2\), we obtain
\[
\Phi
= D(B_{1\bar1},B_{2\bar2})-D(B_{1\bar2},B_{2\bar1})
=\int_{S^3}\bigl(4-3\det(C(\xi))\bigr)\,d\mu(\xi),
\]
where
\[
C(\xi)=
\begin{pmatrix}
\xi^*B_{1\bar1}\xi & \xi^*B_{1\bar2}\xi\\
\xi^*B_{2\bar1}\xi & \xi^*B_{2\bar2}\xi
\end{pmatrix}.
\]
For each \(\xi\in S^3\), the matrix \(C(\xi)\) is Hermitian positive definite, and
\[
\tr(C(\xi))=\xi^*(B_{1\bar1}+B_{2\bar2})\xi=2,
\]
since \(B_{1\bar1}+B_{2\bar2}=2I\). Hence
\[
\det(C(\xi))\le \left(\frac{\tr(C(\xi))}{2}\right)^2=1,
\]
and therefore
\[
\Phi=\int_{S^3}\bigl(4-3\det(C(\xi))\bigr)\,d\mu(\xi)\ge 1.
\]

\subsection{The case of rank three}

We now consider the case \(r=3\). Expanding the definition of \(\Phi\), we obtain
\begin{equation}\label{eq:Phi-6terms}
\begin{aligned}
\Phi
&=
D(B_{1\bar1},B_{2\bar2},B_{3\bar3})
+D(B_{1\bar2},B_{2\bar3},B_{3\bar1})
+D(B_{1\bar3},B_{2\bar1},B_{3\bar2})\\
&\quad
-D(B_{1\bar2},B_{2\bar1},B_{3\bar3})
-D(B_{1\bar3},B_{3\bar1},B_{2\bar2})
-D(B_{2\bar3},B_{3\bar2},B_{1\bar1}).
\end{aligned}
\end{equation}

By Lemma~\ref{lemma:trace-moment} (with \(r=3\)), we have the following spherical moment identities on \(S^{5}=\{\xi\in\mathbb C^{3}:\|\xi\|=1\}\).

\begin{lemma}\label{lem:moments}
For any \(3\times 3\) matrices \(U,V,W\), one has
\[
12\int_{S^5}(\xi^{*}U\xi)(\xi^{*}V\xi)\,d\mu(\xi)
=
\tr(U)\tr(V)+\tr(UV),
\]
and
\begin{align*}
60\int_{S^5}(\xi^{*}U\xi)(\xi^{*}V\xi)(\xi^{*}W\xi)\,d\mu(\xi)
&=
\tr(U)\tr(V)\tr(W)
+\tr(UV)\tr(W)+\tr(UW)\tr(V)\\
&\quad +\tr(VW)\tr(U)
+\tr(UVW)+\tr(UWV).
\end{align*}
\end{lemma}

We also recall the standard trace expansion (valid for all \(3\times 3\) matrices \(U,V,W\)):
\begin{equation}\label{eq:trace-expansion}
\begin{aligned}
6D(U,V,W)
&=\tr(U)\tr(V)\tr(W)
-\tr(U)\tr(VW)-\tr(V)\tr(UW)-\tr(W)\tr(UV)\\
&\quad+\tr(UVW)+\tr(UWV).
\end{aligned}
\end{equation}

Combining Lemma~\ref{lem:moments} with \eqref{eq:trace-expansion} yields the following averaging formula for the mixed discriminant.

\begin{lemma}\label{lem:D-average}
For any \(3\times 3\) matrices \(U,V,W\),
\begin{equation}\label{eq:D-average}
\begin{aligned}
6D(U,V,W)
&=
60\int_{S^5}(\xi^{*}U\xi)(\xi^{*}V\xi)(\xi^{*}W\xi)\,d\mu(\xi)\\
&\quad
-2\Big(\tr(UV)\tr(W)+\tr(UW)\tr(V)+\tr(VW)\tr(U)\Big).
\end{aligned}
\end{equation}
\end{lemma}

For \(\xi\in S^{5}\), set
\[
a=\xi^{*}B_{1\bar1}\xi,\quad
b=\xi^{*}B_{2\bar2}\xi,\quad
c=\xi^{*}B_{3\bar3}\xi,\qquad
x=\xi^{*}B_{1\bar2}\xi,\quad
y=\xi^{*}B_{2\bar3}\xi,\quad
z=\xi^{*}B_{1\bar3}\xi.
\]
Define
\[
C(\xi):=(\xi^*B_{i\bar j}\xi)_{1\le i,j\le 3}
=
\begin{pmatrix}
a & x & z\\
\overline{x} & b & y\\
\overline{z} & \overline{y} & c
\end{pmatrix}.
\]
Under assumption (i) in Problem~\ref{problem2}, the matrix \(C(\xi)\) is positive definite for every \(\xi\in S^5\).
Moreover, by (ii) we have
\[
\tr C(\xi)=a+b+c=\xi^*\Bigl(\sum_{i=1}^3 B_{i\bar i}\Bigr)\xi=\xi^*(3I)\xi=3.
\]
A direct computation gives
\begin{equation}\label{eq:detC}
\det C(\xi)=abc+2\Re(xy\overline{z})-c|x|^{2}-b|z|^{2}-a|y|^{2}.
\end{equation}
We also introduce
\begin{equation}\label{eq:sigma2}
\sigma_{2}(C(\xi))
:=
(ab-|x|^{2})+(ac-|z|^{2})+(bc-|y|^{2})
=
ab+ac+bc-(|x|^{2}+|y|^{2}+|z|^{2}).
\end{equation}

\begin{lemma}\label{lem:Phi-integral}
Under assumptions \emph{(i)}--\emph{(iii)} in Problem~\ref{problem2} with \(r=3\), one has
\begin{equation}\label{eq:Phi-integral}
\Phi
=
\int_{S^{5}}
\Bigl(
10\,\det C(\xi)+27-12\,\sigma_{2}(C(\xi))
\Bigr)\,d\mu(\xi).
\end{equation}
\end{lemma}

\begin{proof}
Using \eqref{eq:detC} and Lemma~\ref{lem:moments}, we compute
\begin{align*}
&\quad 60\int_{S^5}\det C(\xi)\,d\mu(\xi)
=
60\int_{S^5}\bigl(abc+2\Re(xy\overline{z})-c|x|^{2}-b|z|^{2}-a|y|^{2}\bigr)\,d\mu(\xi)\\
&=6\Phi
+6\Bigl[
\tr(B_{1\bar1}B_{2\bar2}-B_{1\bar2}B_{2\bar1})
+\tr(B_{2\bar2}B_{3\bar3}-B_{2\bar3}B_{3\bar2})
+\tr(B_{1\bar1}B_{3\bar3}-B_{1\bar3}B_{3\bar1})
\Bigr].
\end{align*}
On the other hand, by \eqref{eq:sigma2} and Lemma~\ref{lem:moments},
\begin{align*}
&\quad 12\int_{S^5}\sigma_2(C(\xi))\,d\mu(\xi)
=
\bigl(\tr(B_{1\bar1})\tr(B_{2\bar2})+\tr(B_{1\bar1})\tr(B_{3\bar3})+\tr(B_{2\bar2})\tr(B_{3\bar3})\bigr)\\
&\quad
+\Bigl[
\tr(B_{1\bar1}B_{2\bar2}-B_{1\bar2}B_{2\bar1})
+\tr(B_{2\bar2}B_{3\bar3}-B_{2\bar3}B_{3\bar2})
+\tr(B_{1\bar1}B_{3\bar3}-B_{1\bar3}B_{3\bar1})
\Bigr].
\end{align*}
By assumption {(iii)}, \(\tr(B_{i\bar i})=3\) for \(i=1,2,3\), hence the first line equals \(3\cdot 3+3\cdot 3+3\cdot 3=27\).
Subtracting the two identities above yields \eqref{eq:Phi-integral}.
\end{proof}

We will use the following classical inequality (often attributed to Schur).

\begin{lemma}[Schur inequality]\label{lem:Schur}
If \(\lambda_1,\lambda_2,\lambda_3\ge 0\), then
\[
(\lambda_1+\lambda_2+\lambda_3)^3+9\lambda_1\lambda_2\lambda_3
\;\ge\;
4(\lambda_1+\lambda_2+\lambda_3)(\lambda_1\lambda_2+\lambda_1\lambda_3+\lambda_2\lambda_3),
\]
with equality if and only if \(\lambda_1=\lambda_2=\lambda_3\), or two of \(\lambda_1,\lambda_2,\lambda_3\) are equal and the third is \(0\).
\end{lemma}

\begin{proof}
Set \(a=\lambda_1\), \(b=\lambda_2\), \(c=\lambda_3\) and consider
\[
\Delta:=(a+b+c)^3+9abc-4(a+b+c)(ab+ac+bc).
\]
A direct rearrangement gives
\[
\Delta=a(a-b)(a-c)+b(b-c)(b-a)+c(c-a)(c-b).
\]
By symmetry we may assume \(a\ge b\ge c\ge 0\). Then
\[
\Delta=(a-b)^2(a+b-c)+c(a-c)(b-c)\ge 0,
\]
which proves the inequality. The equality statement follows from \(\Delta=0\) under the same ordering. The proof is complete.
\end{proof}

\begin{thm}\label{thm:r=3}
Under assumptions \emph{(i)}--\emph{(iii)} in Problem~\ref{problem2} with \(r=3\), one has
\[
\Phi\ge \int_{S^5}\det C(\xi)\,d\mu(\xi)>0.
\]
\end{thm}

\begin{proof}
Fix \(\xi\in S^{5}\), and let \(\lambda_1,\lambda_2,\lambda_3>0\) be the eigenvalues of the positive definite matrix \(C(\xi)\). Then
\[
\lambda_1+\lambda_2+\lambda_3=\tr (C(\xi))=3,\qquad
\det C(\xi)=\lambda_1\lambda_2\lambda_3,\qquad
\sigma_2(C(\xi))=\lambda_1\lambda_2+\lambda_1\lambda_3+\lambda_2\lambda_3.
\]
Applying Lemma~\ref{lem:Schur} with \(\lambda_1+\lambda_2+\lambda_3=3\) yields
\[
27+9\det C(\xi)-12\,\sigma_2(C(\xi))\ge 0.
\]
Therefore,
\[
10\det C(\xi)+27-12\,\sigma_2(C(\xi))
=
\det C(\xi)
+\bigl(27+9\det C(\xi)-12\,\sigma_2(C(\xi))\bigr)
\ge \det C(\xi).
\]
Integrating over \(S^{5}\) and using Lemma~\ref{lem:Phi-integral} gives
\[
\Phi\ge \int_{S^5}\det C(\xi)\,d\mu(\xi)>0,
\]
since \(\det C(\xi)>0\) for all \(\xi\in S^5\).
\end{proof}

\begin{rem}
When \(r=4\), it is natural to try to follow the same strategy as in the lower-rank cases, namely to express \(\Phi\) as an integral over \(S^{7}\) of a symmetric polynomial in the eigenvalues of the Hermitian matrix
\[
C(\xi):=\bigl(\xi^{*}B_{i\bar j}\xi\bigr)_{1\le i,j\le 4}.
\]
Using Lemma~\ref{lemma:trace-moment}, one can indeed show that
\begin{align*}
\Phi
&=\int_{S^7}\Bigl(35\,\sigma_4(C(\xi))+128-\tfrac{80}{3}\,\sigma_2(C(\xi))\Bigr)\,d\mu(\xi)\\
&\quad
-\tfrac{1}{12}\sum_{\sigma\in S_4}\sgn(\sigma)\,
Q\bigl(B_{1\overline{\sigma(1)}},B_{2\overline{\sigma(2)}},B_{3\overline{\sigma(3)}},B_{4\overline{\sigma(4)}}\bigr),
\end{align*}
where \(\sigma_k(C(\xi))=\sum_{i_1<\cdots<i_k}\lambda_{i_1}\cdots\lambda_{i_k}\) denotes the \(k\)-th elementary symmetric polynomial in the eigenvalues \(\lambda_1,\dots,\lambda_4\) of \(C(\xi)\), and
\[
Q(U_1,U_2,U_3,U_4)
:=\sum_{\rho\in\mathcal{C}_4}\tr\bigl(U_{\rho(1)}U_{\rho(2)}U_{\rho(3)}U_{\rho(4)}\bigr),
\]
with \(\mathcal{C}_4\subset S_4\) the set of \(4\)-cycles.

The difficulty is that the remaining term
\[
-\frac{1}{12}\sum_{\sigma\in S_4}\sgn(\sigma)\,
Q\bigl(B_{1\overline{\sigma(1)}},B_{2\overline{\sigma(2)}},B_{3\overline{\sigma(3)}},B_{4\overline{\sigma(4)}}\bigr)
\]
does not seem to admit a representation as an integral over \(S^7\) of a symmetric polynomial in \(C(\xi)\) alone. This obstruction makes the purely integral approach significantly harder to implement in rank \(4\).
\end{rem}

By Theorem~\ref{thm:r=3}, we obtain an affirmative answer to Problem~\ref{problem2} in the case \(r=3\). Consequently, Problem~\ref{problem1} is also settled for \(r=3\), since the two problems are equivalent for each fixed rank \(r\). By \cite[Corollary~3.5]{Fin}, it follows that the top Chern form
\[
P_{(3,0,0)}\bigl(c(E,h^E)\bigr)=c_3(E,h^E)
\]
is weakly positive for every Griffiths positive Hermitian holomorphic vector bundle \((E,h^E)\) of rank three.

On the other hand,  note that Fagioli \cite[Theorem~0.1 and Theorem~2.2]{Fag20} proved that \(P_{(2,1,0)}\bigl(c(E,h^E)\bigr)\) and $c_2(E,h^E)$ is weakly nonnegative for Griffiths semipositive vector bundles of rank three. We now establish weak positivity in the Griffiths positive case.

\begin{prop}\label{prop:P210-positive}
Let \((E,h^E)\to X\) be a Griffiths positive Hermitian holomorphic vector bundle of rank three over a complex manifold \(X\). Then the Schur forms
\[
P_{(2,1,0)}\bigl(c(E,h^E)\bigr)=c_1(E,h^E)\wedge c_2(E,h^E)-c_3(E,h^E)
\]
and 
\begin{equation*}
  P_{(2,0,0)}\bigl(c(E,h^E)\bigr)=c_2(E,h^E)
\end{equation*}
are weakly positive.
\end{prop}

\begin{proof}
Fix a point \(x\in X\). We prove that \(P_{(2,1,0)}\bigl(c(E,h^E)\bigr)\) is weakly positive at \(x\).
Choose a holomorphic coordinate neighborhood \(U\ni x\) with coordinates \(z=(z^1,\dots,z^n)\) such that \(z(x)=0\), where \(n=\dim X\).
Let \(L\cong \mathbb C\) be the trivial line bundle over \(U\) equipped with the Hermitian metric
\[
h^L_\varepsilon=e^{\varepsilon |z|^2},
\quad
|z|^2=\sum_{p=1}^n |z^p|^2,
\]
for \(\varepsilon>0\) small. Then
\[
c_1(L,h^L_\varepsilon)
=\frac{\sqrt{-1}}{2\pi}\bar\partial\partial\log h^L_\varepsilon
=-\varepsilon\,\omega,
\quad
\omega:=\frac{\sqrt{-1}}{2\pi}\sum_{p=1}^n dz^p\wedge d\bar z^p,
\]
and \(\omega\) is strongly positive.

Let \(R^E\) be the Chern curvature of \((E,h^E)\). The total Chern form of \((E\otimes L,\,h^E\otimes h^L_\varepsilon)\) admits the expansion
\begin{align*}
c(E\otimes L,h^E\otimes h^L_\varepsilon)
&=\det\!\left(I+\Bigl(\tfrac{\sqrt{-1}}{2\pi}R^E-\varepsilon\,\omega\otimes I\Bigr)\right)\\
&=1+\bigl(c_1(E,h^E)-3\varepsilon\,\omega\bigr)\\
&\quad+\bigl(c_2(E,h^E)-2\varepsilon\,\omega\wedge c_1(E,h^E)+3\varepsilon^2\,\omega^2\bigr)\\
&\quad+\bigl(c_3(E,h^E)-\varepsilon\,\omega\wedge c_2(E,h^E)
+\varepsilon^2\,\omega^2\wedge c_1(E,h^E)-\varepsilon^3\,\omega^3\bigr).
\end{align*}
Consequently,
\begin{align*}
&\quad\ c_1(E\otimes L,h^E\otimes h^L_\varepsilon)\wedge c_2(E\otimes L,h^E\otimes h^L_\varepsilon)
-c_3(E\otimes L,h^E\otimes h^L_\varepsilon)\\
&=c_1(E,h^E)\wedge c_2(E,h^E)-c_3(E,h^E)
-2\varepsilon\,\omega\wedge\bigl(c_2(E,h^E)+c_1(E,h^E)^2\bigr)
+O(\varepsilon^2),
\end{align*}
and hence
\begin{equation}\label{eqn5}
P_{(2,1,0)}\bigl(c(E,h^E)\bigr)
=
P_{(2,1,0)}\bigl(c(E\otimes L,h^E\otimes h^L_\varepsilon)\bigr)
+2\varepsilon\,\omega\wedge\bigl(c_2(E,h^E)+c_1(E,h^E)^2\bigr)
+O(\varepsilon^2).
\end{equation}

Since \((E,h^E)\) is Griffiths positive, there exists \(\varepsilon_0>0\) such that \((E\otimes L,h^E\otimes h^L_\varepsilon)\) remains Griffiths positive on a possibly smaller neighborhood of \(x\) for all \(0<\varepsilon<\varepsilon_0\).
By \cite[Theorem~0.1]{Fag20}, the Schur form \(P_{(2,1,0)}\bigl(c(E\otimes L,h^E\otimes h^L_\varepsilon)\bigr)\) is weakly nonnegative.
Moreover, \(\omega\wedge c_2(E,h^E)\) is weakly positive, while \(\omega\wedge c_1(E,h^E)^2\) is strongly positive.
Therefore the term
\[
2\varepsilon\,\omega\wedge\bigl(c_2(E,h^E)+c_1(E,h^E)^2\bigr)
\]
is weakly positive for every \(\varepsilon>0\), and it dominates the error \(O(\varepsilon^2)\) in \eqref{eqn5} when \(\varepsilon\) is sufficiently small.
It follows that \(P_{(2,1,0)}\bigl(c(E,h^E)\bigr)\) is weakly positive at \(x\), and hence on \(X\).

Similarly, note that
\[
c_2(E\otimes L,\,h^E\otimes h^L_{\varepsilon})
=
c_2(E,h^E)-2\varepsilon\,\omega\wedge c_1(E,h^E)+3\varepsilon^2\,\omega^2.
\]
By the same argument, applying \cite[Theorem~2.2]{Fag20} and choosing \(\varepsilon>0\) sufficiently small, we conclude that
\(c_2(E,h^E)\) is weakly positive.
\end{proof}

\begin{proof}[Proof of Theorem~\ref{thm:main}]
The inequality \(\Phi>0\) follows immediately from Theorem~\ref{thm:r=3}. In particular, this settles Open Problem~\ref{problem2} in the case \(r=3\). By the reduction explained in Section~\ref{Sec: Fin-open}, it follows that, for every positive linear map \(H:\End(V)\to \End(W)\) with \(\dim V=\dim W=r=3\), the real number
\[
\mathrm{D}_W\circ H^{\otimes \dim W}\circ \mathrm{D}_V^*(1)
\]
is positive.

Now let \((E,h^E)\) be a Griffiths positive Hermitian holomorphic vector bundle of rank three over a complex threefold. The Schur forms in rank three correspond to the partitions
\[
(0,0,0),\ (1,0,0),\ (2,0,0),\ (1,1,0),\ (3,0,0),\ (2,1,0),\ (1,1,1).
\]
By \cite[Corollary~3.5]{Fin}, the positivity of \(\mathrm{D}_W\circ H^{\otimes 3}\circ \mathrm{D}_V^*\) implies that the top Chern form
\[
P_{(3,0,0)}\bigl(c(E,h^E)\bigr)=c_3(E,h^E)
\]
is weakly positive. Moreover, G\"uler \cite[Theorem~1.1]{MR2932990} proved that all signed Segre forms are weakly positive; in particular,
\[
P_{(1,0,0)}\bigl(c(E,h^E)\bigr),\quad
P_{(1,1,0)}\bigl(c(E,h^E)\bigr),\quad
P_{(1,1,1)}\bigl(c(E,h^E)\bigr)
\]
are weakly positive. Finally, Proposition~\ref{prop:P210-positive} shows that
\[
P_{(2,0,0)}\bigl(c(E,h^E)\bigr)\quad\text{and}\quad
P_{(2,1,0)}\bigl(c(E,h^E)\bigr)
\]
are weakly positive. (Here \(P_{(0,0,0)}\bigl(c(E,h^E)\bigr)=1\).)
Therefore, all Schur forms in rank three are weakly positive, and the proof is complete.
\end{proof}

\bibliographystyle{alpha}
\bibliography{posi}

\end{CJK}
\end{document}